\documentclass[12pt]{amsart}
\usepackage{amsthm}
\usepackage{amssymb}
\usepackage{amsmath}
\usepackage{tikz-cd}
\usepackage{color}

\newtheorem{theorem}{Theorem}[section]
\newtheorem{proposition}[theorem]{Proposition}

\newtheorem{corollary}[theorem]{Corollary}
\theoremstyle{definition}
\newtheorem{definition}[theorem]{Definition}
\newtheorem{example}[theorem]{Example}

\newtheorem{remark}[theorem]{Remark}
\theoremstyle{remark}

\newcommand{\brref}[1]{(\ref{#1})}
\newcommand{\bslocus}[1]{{\rm Bs} |#1|}
\newcommand{\bslocusnobars}[1]{{\rm Bs}\, #1 }
\newcommand{\calo}{{\mathcal O}}
\newcommand{\milano}{Dipartimento di Matematica ``F. Enriques"
	\\ Universit\`a degli Studi di Milano \\ Via Saldini 50 \\ 20133
	Milano, Italy}
\newcommand\zed{{\mathbb Z}}
\newcommand\nat{{\mathbb N}}
\newcommand\rat{{\mathbb Q}}
\newcommand\real{{\mathbb R}}
\newcommand\comp{{\mathbb C}}
\newcommand{\Pin}[1]{{\mathbb P}^{#1}}
\newcommand{\restrict}[2]{{#1}_{\mid _{#2}}}
\newcommand{\lra}{\longrightarrow}
\newcommand{\ra}{\rightarrow}
\newcommand{\Proj}[1]{\mathbb{P}(#1)}
\newcommand{\Projcal}[1]{\mathbb{P}({\mathcal #1})}

\newcommand{\oof}[2]{{\mathcal O}_{#1}({#2})}
\newcommand{\oofp}[2]{{\mathcal O}_{\mathbb{ P}^{#1}}({#2})}
\newcommand{\oofo}[1]{{\mathcal O}_{#1}}
\newcommand{\iof}[2]{{\mathcal I}_{#1}({#2})}
\newcommand{\iofo}[1]{{\mathcal I}_{#1}}
\newcommand{\ideal}{\mathcal{I}}
\newcommand{\scroll}[1]{(\Proj{#1},\taut{#1})}
\newcommand{\scrollcal}[1]{(\Projcal{#1},\tautcal{#1})}
\newcommand{\xel} {(X, L)}
\newcommand{\supp}[1]{{\rm Supp}(#1)}
\newcommand{\taut}[1]{{\mathcal O}_{\mathbb{P}(#1)}(1)}
\newcommand{\tautof}[2]{{\mathcal O}_{\mathbb{P}(#1)}(#2)}
\newcommand{\tautcal}[1]{{\mathcal O}_{\mathbb{P}({\mathcal#1})}(1)}
\newcommand{\tautcalof}[2]{{\mathcal O}_{\mathbb{P}({\mathcal#1})}(#2)}
\newcommand{\num}{\equiv}
\newcommand{\elbar}{\overline{\ell}}
\newcommand{\rk}[1]{{\rm rk}\,(#1)}
\newcommand{\lineq}{\sim}
\newcommand{\tauttilde}{\taut{\tilde{E}}}
\newcommand{\matlab}{$\rm{MATLAB}^{\circledR}$}
\newcommand{\maple}{$\rm{Maple}^{\circledR}$}
\newcommand{\macaulay}{$\rm{Macaulay2}^{\copyright}$}
\newcommand{\Ppp}{P_{(\varphi,\psi)}}
\newcommand{\bX}{\{\mathbf{X}\}}
\newcommand{\bxi}{\{\mathbf{X}_i\}}
\newcommand{\bxj}{\{\mathbf{X}_j\}}
\newcommand{\nbXi}{\mathbf{X}_i}
\newcommand{\nbxj}{\mathbf{X}_j}
\newcommand{\critloc}[1]{\mathfrak{C}^#1_{(P,Q)}}
\newcommand{\ucritloc}[1]{\mathcal{U}^#1}
\newcommand{\critlocn}[1]{\mathfrak{C}^#1_{(P_i,Q_i)}}
\newcommand{\zeros}[2]{\mathbf{0}(#1,#2)}
\newcommand{\arow}[2]{\mathbf{#1}^{#2}}
\newcommand{\prow}[2]{\mathbf{p}^#2_#1}
\newcommand{\pxrow}[2]{\mathbf{px}^#2_#1}
\newcommand{\prowx}[2]{\mathbf{p}^#2_#1 \mathbf{X}}
\newcommand{\prowxsub}[3]{\mathbf{p}^#3_{#1_{#2}} \mathbf{X}}
\newcommand{\qrow}[2]{\mathbf{q}^#2_#1}
\newcommand{\qrowx}[2]{\mathbf{q}^#2_#1 \cdot \mathbf{X}}
\newcommand{\q}{\mathbf{q}}
\newcommand{\jvariation}{(j_1, \dots j_n)\in \mathcal{A}^{\times n}}
\newcommand{\Grande}{G^{a_1,\dots,\hat{a_h},\dots,a_n}_{h,s}}

\begin{document}
	
	\date{\today}
	\title[Birational geometry of critical loci]{Birational geometry of critical loci in Algebraic Vision}
	\author[M.Bertolini]{Marina Bertolini$^\ast$} \thanks{$^\ast$ corresponding author}
	\email{marina.bertolini@unimi.it} \thanks{The authors are members
		of GNSAGA of INdAM}
	
	\author[R.Notari]{Roberto Notari}
	\email{roberto.notari@polimi.it}
	
	\author[C.Turrini]{Cristina Turrini}
	\email{cristina.turrini@unimi.it}
	
	\address{\milano}
	\address{Dipartimento di Matematica, Politecnico di Milano, Piazza Leonardo da Vinci 32, 20133 Milano}

\begin{abstract}
In Algebraic Vision, the projective reconstruction of the position of each camera and scene point from the knowledge of many enough corresponding points in the views is called the structure from motion problem. It is known that the reconstruction is ambiguous if the scene points are contained in particular algebraic varieties, called critical loci. To be more precise, from the definition of criticality, for the same reconstruction problem, two critical loci arise in a natural way. In the present paper, we investigate the relations between these two critical loci, and we prove that, under some mild smoothness hypotheses, (some of) their irreducible components are birational. To this end, we introduce a unified critical locus that restores the symmetry between the two critical loci, and a natural commutative diagram relating the unified critical locus and the two single critical loci. For technical reasons, but of interest in its own, we also consider how a critical locus change when one increases the number of views.
\end{abstract}
\maketitle
%

\section{Introduction}

Algebraic Vision is the branch of Computer Vision and Multiview Geometry that addresses the problems with techniques from algebraic geometry. In this context, a classic problem is the {\it structure from motion}, i.e. the projective reconstruction of static and dynamic scenes from photos or videos. More precisely, this problem involves determining the unknown positions both of cameras and of scene points in the ambient space, given a sufficient number of images. Historically, the problem of structure from motion originated with the reconstruction of static scenes in $\mathbb{P}^{3}$ from images in $\mathbb{P}^{2}$ (\cite{Hart-Zi2})  and was later generalized to the reconstruction of dynamic scenes in $\mathbb{P}^{3}$ modelled as static scenes in ambient spaces of any dimension (\cite{WS,survey}).

In this setting, cameras are identified with projection matrices between projective spaces of any dimension.

Naively, one might guess that the reconstruction problem can be successfully solved if a sufficient number of $n$--tuples of corresponding points, that is to say, points which are projections of the same point of the scene, are identified in different views. Indeed, even if this happens when cameras and scene points are in sufficiently general position, this is not always the case. E.g., even in the classical setup of two projections from three-dimensional space, for particular configurations of points $ \mathbf{X}_i$ in the scene and cameras ${Q}_j$ there exist another - non-projectively equivalent - sets of scene points $ \mathbf{Y}_i$ and cameras ${P}_j$ producing the same images in the view planes, so leading to ambiguous reconstruction. The locus that these scene points describe is called {\it critical} with respect to the corresponding set of cameras.

Critical loci turn out to be determinantal varieties in the projective spaces. They have been extensively studied, both in the classical case of any number of projections from $\mathbb{P}^{3}$ to $\mathbb{P}^{2}$ (\cite{Hart-Ka}), and in the case of $n$ projections from $\mathbb{P}^{k}$ to $\mathbb{P}^{h_i}, i= 1, \dots, n$ (\cite{bbnt,bnt1,bnt2,tubbLAIA,bnt3,survey}).

From the definition of the critical locus (see Definition \ref{critconfm}), it naturally emerges that critical points actually occur in pairs $(\mathbf{X}_i,\mathbf{Y}_i)$, with $ \mathbf{X}_i \in \mathcal{X}$ and $ \mathbf{Y}_i \in \mathcal{Y}$. The purpose of this work is to determine, from the perspective of birational geometry, the relationship between the critical loci $\mathcal{X}$ and $\mathcal{Y}$. To this purpose, we introduce the notion of {\it unified critical locus} $\mathcal{U} $ whose elements are the previous couples of critical points $ (\mathbf{X}_i,\mathbf{Y}_i)$.

For a set of $n$ projections from  $\mathbb{P}^{k}$ to $\mathbb{P}^{h_i}$, $i=1, \dots, n$, the unified critical locus is an algebraic variety contained in $\mathbb{P}^{k} \times \mathbb{P}^{k}$ for which we get the defining ideal in Section \ref{sec::crit-loc}. Hence, we have a natural commutative diagram linking the unified critical locus and the two critical loci $\mathcal{X},\mathcal{Y}$ (diagram \ref{main-cd}) which is extensively studied in Section \ref{maps-sect}. The main result is Theorem \ref{birational}, where we prove that there are irreducible components inside each critical locus, that are birational under some reasonable smoothness assumptions. On the other side, Section \ref{sec::examples} contains some explicit examples showing that the hypotheses in Theorem \ref{birational} cannot be relaxed.

A fundamental tool for obtaining the results in Section \ref{maps-sect}, and of interest in its own right, is a preliminary investigation into another aspect of the construction of critical loci. Specifically, we investigate how critical loci change when we increase the number of projections of the same scene. Section \ref{sec::nestingcritical loci} is dedicated to this analysis. A seminal case of nested critical loci is considered in \cite{Hart-Ka}, in the case of two and three projections from $\mathbb{P}^{3}$ to $\mathbb{P}^{2}$.

Finally, in the first part of Section \ref{sec::crit-loc}, we recall the standard set--up in multiview geometry, and the main definitions and results on critical loci; in the second part of the same section, we define the unified critical locus, describe its defining ideal and the commutative diagram that relates $ \mathcal{U}, \mathcal{X} $ and $ \mathcal{Y}$.


\section{Multiview geometry and critical loci}
\label{sec::crit-loc}
In the first subsection, we recall the standard set--up used when studying Computer Vision or better Multiview Geometry problems with techniques from algebraic geometry. Since the literature on the subject is becoming reacher and reacher, nowadays it goes under the name of Algebraic Vision. Moreover, we recall the definition of critical locus, the construction of its defining ideal, and the expected dimension and degree. In the second part, we define the unified critical locus, explain the construction of its defining ideal, and the natural commutative diagram relating the unified critical locus and the single critical loci.

\subsection{Notation and critical locus definition}

In Algebraic Vision, a {\it camera} $P$ is a linear projection from
$\mathbb{P}^{k}$ onto $\mathbb{P}^{h},$  from a linear subspace $ C $ of
dimension $ k-h-1$,  called {\it center of projection}. The target
space $\mathbb{P}^{h}$ is called {\it view}. A {\it scene} is a set of
points $ \mathbf{X}_i \in \mathbb{P}^{k}, i=1, \dots, N$.

Using homogeneous coordinates in $\mathbb{P}^{k}$ and $\mathbb{P}^{h},$ we
identify $P$ with a $(h+1) \times (k+1) $ matrix of maximal rank,
defined up to a multiplicative constant.  Hence, $ C $ comes out to
be the right annihilator of $P$.

Let us consider a set of $ n $ cameras $P_j:\mathbb{P}^{k}\setminus C_{j} \to
\mathbb{P}^{h_j}$, $j=1, \dots, n$, projecting the same scene in $\mathbb{P}^k$. The images $P_j(\mathbf{X}) \in \mathbb{P}^{h_j} $ of the same scene point $\mathbf{X} \in \mathbb{P}^{k}$ in the different views are said to be \textit{corresponding points}.

As we said in the Introduction, a classical problem in this context is the so--called
{\it{projective reconstruction}} of a scene: determine $ n $ cameras $P_j:\mathbb{P}^{k}\setminus C_{j} \to
\mathbb{P}^{h_j}$, $ j=1, \dots, n$ (i.e., the projection matrices), and a scene in
$\mathbb{P}^k$ (i.e., the coordinates of the scene points), starting from many enough corresponding points in the views, up to projective transformations. It is known that the reconstruction is ambiguous when all scene points lie on the critical locus.

In order to describe systematically critical loci, it is appropriate to give formal definitions.
\begin{definition} \label{critconfm}
Given two sets of $ n $ projections $ Q_j, P_j: \mathbb{P}^k \dashrightarrow \mathbb{P}^{h_j}$,$j=1, \dots, n$,
two sets of points $\{ \mathbf{X}_1, \dots, \mathbf{X}_N\}$ and $\{ \mathbf{Y}_1, \dots, \mathbf{Y}_N\}$, $ N \gg 0$,
in $\mathbb{P}^{k}$, are said to be \textit{conjugate critical
configurations}, associated to the projections
$\{Q_1, \dots, Q_n\}$ and $\{P_1, \dots, P_n\}$ if, for all $i = 1, \dots, N$ and $j = 1,
\dots, n$, we have $Q_j(\mathbf{X}_i) = P_j(\mathbf{Y}_i)$.
\end{definition}

Notice that in Definition \ref{critconfm}, the symmetry between the projections $ Q_1, \dots, Q_n $ and $ P_1, \dots, P_n $ is perfect.

Now, we can define the critical loci.
\begin{definition} \label{critical-loci}
Given two sets of $ n $ projections $ Q_j, P_j : \mathbb{P}^k \dashrightarrow \mathbb{P}^{h_j}$, $j=1, \dots, n$, as above, the locus $ \mathcal{X} \subseteq \mathbb{P}^k $ ($ \mathcal{Y} \subseteq \mathbb{P}^k$, respectively) containing all possible critical configurations $ \{ \mathbf{X}_1, \dots, \mathbf{X}_N \} $ ($ \{ \mathbf{Y}_1, \dots, \mathbf{Y}_N \}$, respectively) is called \textit{critical locus} for the associated projections.
\end{definition}

The current definition is slightly different from the analogous one in literature, e.g. see \cite{bnt1,bnt2,Hart-Ka}. We modify it in order to be closer to the standard practice when computing critical loci. The difference is that here we first choose the two $n$--tuples of projections, and we consider two scenes that are made of critical points, while, in the old one, one has to consider a $n$--tuple of projections and its scene, and in a second stage, one has to find a second $n$--tuple of projections and a second scene such that the same condition of the definition apply.

As previously said, if the scene points do not belong to a critical locus, the reconstruction problem has a unique solution, or equivalently, is not ambiguous.

We recall how to get explicitly the defining ideal of the critical loci for a pair of $n$--tuples of conjugate projections $ \{Q_1, \dots,Q_n\}$ and  $ \{P_1, \dots,P_n\}$. To easier distinguish between $ \mathcal{X} $ and $ \mathcal{Y}$, we use different names for the coordinates in $ \mathbb{P}^k$: when computing $ \mathcal{X}$, we use $ x_0, \dots, x_k$, while we use $ y_0, \dots, y_k$ when we compute $ \mathcal{Y}$.

In \cite{Hart-Schaf} the authors define a family of tensors, the so-called {\it{Grassmann tensors}}, which encode the constraints existing among corresponding points in the different images. In \cite{bnt2}, making use of these tensors, the authors show that the critical locus $ \mathcal{X} \subseteq \mathbb{P}^k $ is a determinantal variety and that its ideal $ I(\mathcal{X}) $ can be obtained as follows.

Let $\mathbf{X}$ be a point in $\mathbb{P}^k$ and consider the matrix \begin{equation} \label{matrix-MX}  M_{\mathcal{X}} = \left( \begin{array}{ccccccc} P_1 &
Q_1(\mathbf{X}) & 0 & 0 & \dots & 0 & 0 \\ P_2 & 0 &
Q_2(\mathbf{X}) & 0 & \dots & 0 & 0 \\ \vdots & & & & & & \vdots
\\ P_n & 0 & 0 & 0 & \dots & 0 & Q_n(\mathbf{X}) \end{array}
\right).\end{equation}
$ M_{\mathcal{X}}$ is a $ (n+ \displaystyle \sum_{i=1}^n h_i) \times (n+k+1) $ matrix,
the last $n$ columns of which are of linear forms, while the first
$ k+1 $ columns are of constants.

\begin{proposition} \label{prop-2-1} (\cite{bnt2})
	The ideal $ I(\mathcal{X})$ of the critical locus $ \mathcal{X} $
	is generated by the maximal minors of $ M_{\mathcal{X}}$, and so $\mathcal{X}$ is a determinantal variety.
	Moreover, $ \mathcal{X} $ contains the center of $ Q_j$, $j=1,\dots, n$.
\end{proposition}

In a similar way, to get the critical locus $ \mathcal{Y} \subseteq \mathbb{P}^k$, that has the same properties of $ \mathcal{X}$, one has to use the matrix

\begin{equation} \label{matrix-MY}  M_{\mathcal{Y}} = \left( \begin{array}{ccccccc} Q_1 &
P_1(\mathbf{Y}) & 0 & 0 & \dots & 0 & 0 \\ Q_2 & 0 &
P_2(\mathbf{Y}) & 0 & \dots & 0 & 0 \\ \vdots & & & & & & \vdots
\\ Q_n & 0 & 0 & 0 & \dots & 0 & P_n(\mathbf{Y}) \end{array}
\right).\end{equation}

Moreover, from the determinantal nature of the critical loci $ \mathcal{X} $ and $ \mathcal{Y}$, one also gets that their expected dimension is
\begin{equation} \label{exp-dim}
k - \left(1 + (n -k-1+ \sum_{i=1}^n h_i) -
n\right) = 2k - \sum_{i=1}^n h_i.
\end{equation}
We explicitly remark that, for particular choices of the $n$--tuples of projections, the dimension of $ \mathcal{X}$, or of $ \mathcal{Y}$, or of both, can be greater than the expected, e.g. see \cite{Hart-Ka,bbnt}. Finally, when the critical locus has the expected dimension, its degree is
\begin{equation} \label{exp-deg}
\binom{n-k-1+\sum_{i=1}^n h_i}{n-1}.
\end{equation} \bigskip

\subsection{The unified critical locus}
\label{sec::unif-crit-loc}

Now, we introduce the unified critical locus to restore the symmetry between the two sets of projections.
\begin{definition} \label{def-unif-crit}
Given two sets of $ n $ projections $ Q_j, P_j : \mathbb{P}^k \dashrightarrow \mathbb{P}^{h_j}$, $j=1, \dots, n$, as above, the unified critical locus is
\[
\mathcal{U} = \{ (\mathbf{X}, \mathbf{Y}) \in \mathbb{P}^k \times \mathbb{P}^k \ \vert \ \mathbf{X} \in \mathcal{X}, \mathbf{Y} \in \mathcal{Y}, Q_j(\mathbf{X}) = P_j(\mathbf{Y}), j=1,\dots, n\},
\]
where $ \mathcal{X}, \mathcal{Y} $ are the critical loci associated to the $ n $ projections.
\end{definition}

From its definition, it follows that $ \mathcal{U} $ is an algebraic variety in $ \mathbb{P}^k \times \mathbb{P}^k$.
\begin{proposition} \label{prop-id-unif-crit-locus}
The defining ideal $ I(\mathcal{U}) $ is bi--homogeneous generated by the maximal minors of $ M_{\mathcal{X}}$, $ M_{\mathcal{Y}} $ and by the $ 2 \times 2 $ minors of the matrices $$ \left( \begin{array}{ccc} P_j(\mathbf{Y}) & \vert & Q_j(\mathbf{X}) \end{array} \right), \quad j=1, \dots, n.$$
\end{proposition}

\begin{remark} \rm
The construction of the unified critical locus reminds the blow--up of $ \mathbb{P}^k $ at the centers of the $n$ projections $ Q_1, \dots, Q_n$, at least when the projection centers are pairwise disjoint. The main difference is that we do not use independent variables for each center, so that we get an image of the blow--up in $ \mathbb{P}^k$.
\end{remark}

The unified critical locus $ \mathcal{U} $ is naturally equipped with two projections to the critical loci $ \mathcal{X}, \mathcal{Y}$, and so we have the diagram
\begin{equation}\label{main-cd} \begin{tikzcd}  & \mathcal{U} \ar[dl, "\pi_1"'] \ar[dr, "\pi_2"] & \\ \mathcal{X} & & \mathcal{Y} \end{tikzcd}. \end{equation}

As said before, the main scope of the present paper is to study the birational geometry of the critical loci as it comes from the diagram above. To this end, we will describe the images of the maps $ \pi_1, \pi_2$, and the compositions $ \pi_2 \circ \pi_1^{-1} $ and $ \pi_1 \circ \pi_2^{-1}$. To do this, we must first understand how a critical locus changes when projections of the same scene are added or removed. This will be addressed in the following section.


\section{Nested critical loci} \label{sec::nestingcritical loci}

In this section, we study the relationship between the critical locus $ \mathcal{X}_r $ associated to the projections $ P_1, \dots, P_r $ and $ Q_1, \dots, Q_r$, and the critical locus $ \mathcal{X}_n $ associated to the projections $ P_1, \dots, P_r, \dots, P_n $ and $ Q_1, \dots, Q_r, \dots, Q_n$, for some $ n > r$.

\begin{proposition} \label{prop-nest1}
Let $ \mathcal{X}_r $ be the critical locus associated to $ P_1, \dots, P_r$ and $ Q_1, \dots, Q_r$, and let $ \mathcal{X}_n $ be the critical locus associated to the couple of projections $ P_1, \dots,P_r, \dots, P_n$ and $ Q_1, \dots, Q_r, \dots, Q_n$ for some $ n > r$. Then, set--theoretically, we have $ \mathcal{X}_n \setminus (C_{r+1} \cup \dots \cup C_n) \subseteq \mathcal{X}_r$, while, scheme--theoretically, we have $ I(\mathcal{X}_r) \subseteq I(\mathcal{X}_n) : (I(C_{r+1}) \cdots I(C_n))$, where $ C_j $ is the center of $ Q_j$, for $ j=r+1, \dots, n$.
\end{proposition}

\begin{proof} For the set--theoretical part of the statement, we observe that, since the center $ C_i $ of $ Q_i$, for $ i=1, \dots, r$ is contained in both $ \mathcal{X}_n $ and $ \mathcal{X}_r$, we can take $ \mathbf{X} \in \mathcal{X}_n \setminus (C_1 \cup \dots \cup C_n)$. The matrix $ M_{\mathcal{X}_n}(\mathbf{X}) $ has rank strictly smaller than $ k+1+n$, and so the homogeneous linear system $ M_{\mathcal{X}_n}(\mathbf{X}) Z = 0 $ has non--trivial solutions, with $ Z = (\overline{Z}, z_{k+2}, \dots, z_{k+1+n})^T$. Let $ Z_0 $ be a non--zero solution. $ \overline{Z}_0 $ cannot be zero. Indeed, in case $ \overline{Z}_0 = 0$, then the linear systems above reduce to $ z_{k+1+i} Q_i(\mathbf{X}) = 0 $ for $ i=1, \dots, n$. Since $ \mathbf{X} \notin (C_1 \cup \dots \cup C_n)$, then $ Q_i(\mathbf{X}) \not= 0$ for every $ i$, and so $ z_{k+1+i} = 0$ for every $ i=1, \dots, n$. Then, $ Z_0 = 0 $ that is not possible. Let $ Z' $ be the vector obtained from $ Z_0 $ deleting the last $ n-r $ entries. Then, $ M_{\mathcal{X}_r}(\mathbf{X}) Z' = 0$, and $ Z' $ is non---zero, because of the previous argument. This implies that the rank of $ M_{\mathcal{X}_r}(\mathbf{X}) $ is strictly smaller than $ k+1+r$, and so $ \mathbf{X} \in \mathcal{X}_r$, as claimed.

For the scheme--theoretical part, let us consider an order $ k+1+r $ sub--matrix of $ M_{\mathcal{X}_r}$. If in $  M_{\mathcal{X}_n}$ we take the same rows chosen in $ M_{\mathcal{X}_r}$ as above and, moreover, a row for every other projection $ P_{r+1}, \dots, P_n $ from $ M_{\mathcal{X}_n}$, then we get an order $ k+1+n $ square sub--matrix of $ M_{\mathcal{X}_n}$. By Laplace rule, the corresponding minor in $ M_{\mathcal{X}_n} $ is the product of the one in $ M_{\mathcal{X}_r} $ and an element in $ I(C_{r+1}) \cdots I(C_n)$. Since this holds for every $ k+1+r $ sub--matrix of $ M_{\mathcal{X}_r}$, and every generator of $ I(C_{r+1}) \cdots I(C_n)$, we get the claim.
\end{proof}

\begin{remark} \rm
Under the Computer Vision standard assumption that the intersection between any two centers is empty, we have that
\[
I(\mathcal{X}_{n-1}) \subseteq I(\mathcal{X}_n) : I(C_n)
\]
up to renumbering the projections. This relation is somewhat stronger than the one in Proposition \ref{prop-nest1} because $ I(C_{r+1}) \cdots I(C_n) $ is not the defining ideal of $ C_{r+1} \cup \dots \cup C_n$, in general, for $ r+1 \not= n$. Moreover, when one considers $ n-1 $ projections, then the ideal $ I(\mathcal{X}_n) : I(C_n) $ defines either $ \mathcal{X}_n \setminus C_n$, if $ C_n $ is an irreducible component of $ \mathcal{X}_n $ with a reduced structure, or a scheme $ \mathcal{X}' $ that differs from $ \mathcal{X}_n $ because of the scheme--structure on the irreducible component $ C_n$.

For example, when we consider a couple of three projections from $ \mathbb{P}^3 $ to $ \mathbb{P}^2$, the critical locus is a set of $ 10 $ points, $ 3 $ of which are the projection centers, and the remaining $ 7 $ are in the intersection of the three critical quadrics, each associated to two of the three projections. If we choose the third projection in such a way that its center is contained in the critical quadric of the other two projections, then the critical locus is a set of $ 9 $ points, but the center of the third projection has multiplicity $ 2$ (and so, a non--reduced structure).

However, critical loci can have dimensions larger than the one of the projection centers. E.g, in \cite{Hart-Ka}, there is a classification of critical loci for a couple of three projections from $ \mathbb{P}^3 $ to $ \mathbb{P}^2$, and in particular, the authors show that the critical locus can be a curve, whilst the centers of projection are points.
\end{remark}

The examples in the Remark above can be generalized. In fact, it holds
\begin{proposition} \label{prop-nest2}
Let $ \mathcal{X}_{n-1} $ be the critical locus associated to $ P_1, \dots, P_{n-1}$ and $ Q_1, \dots,$ $ Q_{n-1}$, with $ \sum_{i=1}^{n-1} h_i \geq k+1$, and let $ \mathcal{X}_n $ be the critical locus associated to the previous projections and the further couple $ P_n$ and $ Q_n$.
Let $ \mathbf{X} \in C_{Q_n} \cap  \mathcal{X}_{n-1}$ be a point such that the rank of $ M_{\mathcal{X}_{n-1}}(\mathbf{X}) $ is equal to $ k+n-1$, and the rank of $ M_{\mathcal{X}_n}(\mathbf{X}) $ without the last column is $ k+n$. Then the dimension of the tangent space $T_{\mathbf{X}}(\mathcal{X}_{n})$ is $ k - h_n$.
\end{proposition}

\begin{proof} To prove the result, we compute the rank of the jacobian matrix of $ I(\mathcal{X}_n) $ at $ \mathbf{X} \in C_{Q_n}$. To this end, we set the following notation:
\begin{enumerate}
\item given the length $ k+n $ multi--index $ \underline{\lambda} = (\lambda_1, \dots, \lambda_{k+n})$, with $ 1 \leq \lambda_1 < \dots < \lambda_{k+n} \leq \sum_{i=1}^n h_i + n$, we set $ D_{\underline{\lambda}} $ the determinant of the sub--matrix of $ M_{\mathcal{X}_n} $ obtained by taking the rows $ \lambda_1, \dots, \lambda_{k+n}$ and the first $ k+n $ columns;
\item given the length $ k+n+1 $ multi--index $ \underline{\mu} = (\mu_1, \dots, \mu_{k+n},\mu_{k+n+1})$, with $ 1 \leq \mu_1 < \dots < \mu_{k+n+1} \leq \sum_{i=1}^n h_i + n$, we set $ F_{\underline{\mu}} $ the determinant of the sub--matrix $ M_{\underline{\mu}} $ of $ M_{\mathcal{X}_n} $ obtained by taking the rows $ \mu_1, \dots, \mu_{k+n+1}$ and all columns;
\item given a multi--index $ \underline{\tau}$, denote by $\underline{\tau'}$ the sub  multi--index given by the $ \tau_{i}'s$ such that $\tau_{i} \leq L$, where $ L = \sum_{i=1}^{n-1} h_i + n -1$ is the number of rows of $ M_{\mathcal{X}_{n-1}}$;
\item up to a coordinate change in $ \mathbb{P}^k$, we assume that $Q_n=[I_{h_n+1}|0]$ so that $ C_{Q_n} $ is defined by $ x_0 = \dots = x_{h_n} = 0$;
\item since $ M_{\mathcal{X}_{n-1}}(\mathbf{X}) $ has rank $ k+n-1$ and the matrix obtained from $ M_{\mathcal{X}_n}(\mathbf{X}) $ by canceling the last column has rank $ k+n$, there exists a length $ k+n $ multi--index $ \underline{\lambda}= (\underline{\lambda}', L+1+i)$ such that $ D_{\underline{\lambda}}(\mathbf{X}) \not= 0$, for a suitable $ i \in [0, h_n]$. Without loss of generality, we can assume $ i=0$.
\end{enumerate}

Let $ \mathbf{X} \in C_{Q_n} \cap  \mathcal{X}_{n-1}$ be a point at which the rank of $ M_{\mathcal{X}_{n-1}}(\mathbf{X}) $ is equal to $ k+n-1$, as in the statement. This implies that every minor of $ M_{\mathcal{X}_{n-1}}(\mathbf{X}) $ of order $ k+n $ vanishes, that is to say, $ D_{\underline{\lambda}}(\mathbf{X}) = 0 $ for every $ \underline{\lambda} = (\lambda_1, \dots, \lambda_{k+n}) $ such that $ \lambda_{k+n} \leq L$ i.e. when $ \underline{\lambda} =\underline{\lambda'}$.

We recall (see Proposition \ref{prop-2-1}) that the ideal $ I(\mathcal{X}_n) $ is generated by the order $ k+n+1 $ minors of $ M_{\mathcal{X}_n}$.

Let $ \underline{\mu} = (\mu_1, \dots, \mu_{k+n+1-t}, L+1+i_1, \dots, L+1+i_t) = (\underline{\mu'}, L+1+i_1, \dots, L+1+i_t)  $  where $ i_1, \dots, i_t $ are integers in the range $ [0, h_n]$, in increasing order and let $ F_{\underline{\mu}} $ be the determinant of the sub--matrix $M_{\underline{\mu}} $ of $ M_{\mathcal{X}_n} $ defined above. Of course, $ F_{\underline{\mu}} $ is a generator of $ I(\mathcal{X}_n)$.

In order to determine the rank of the jacobian matrix of $ I(\mathcal{X}_n) $ at $ \mathbf{X}$, we need to analyse the gradient $\nabla F_{\underline{\mu}}(\mathbf{X})$. We recall that the derivative of a determinant can be computed as sum of determinants, where each matrix is obtained by computing the derivative of the polynomials in a column, and by taking all other columns as they are. The derivative of $ F_{\underline{\mu}} $ with respect to a variable different from $ x_0, \dots, x_{h_n}$ evaluated at $ \mathbf{X} $ vanishes, as the last column of every determinant vanishes either because we evaluate it at $ \mathbf{X}$, or because we differentiate it with respect to a variable that does not appear on the last column. On the other end, the derivative of $ F_{\underline{\mu}} $ with respect to one variable among $ x_0, \dots, x_{h_n} $ is equal to the determinant of the matrix we obtain by differentiating the last column of the matrix $ M_{\underline{\mu}} $  with respect to the same variable. Hence, the gradient of $ F_{\underline{\mu}} $ at $ \mathbf{X} $ is
\begin{equation} \label{gradient}
\nabla F_{\underline{\mu}}(\mathbf{X}) = \sum_{j=1}^t (-1)^{t-j} D_{\underline{\mu}\setminus\{L+1+i_j\}}(\mathbf{X}) \vec{e}_{i_j}
\end{equation}
where $ \vec{e}_{0}, \dots, \vec{e}_{h_n} $ are placeholders corresponding to $ x_0, \dots, x_{h_n}$, i.e. they are the first $h_n +1$ vectors of the canonical basis of $\mathbb{C}^{k+1}$.

Of course, for $ t \leq 1$, $ \nabla F_{\underline{\mu}} = \vec{0}$, for every $ \underline{\mu}$. Indeed, for $t=0$, $F_{\underline{\mu}}$ vanishes, whilst, for $t=1,$ $F_{\underline{\mu}}= D_{\underline \mu'} x_j,$ for $ \underline{\mu} = (\underline{\mu'}, L+1+j)$.

\medskip

Our purpose is now to show that the rank of the jacobian matrix of $ I(\mathcal{X}_n) $ at $ \mathbf{X}$ is $h_n.$  We achieve this goal by showing that a basis of the vector space spanned by the gradients of the generators of $ I(\mathcal{X}) $ evaluated at $ \mathbf{X}$ is given by the $h_n$ vectors $\nabla F_{(\underline{\mu}',L+1,L+2)}(\mathbf{X})$, $\dots$, $\nabla F_{(\underline{\mu}',L+1,L+1+h_n)}(\mathbf{X}) $, for a fixed $\underline{\mu}'$. In particular in Claim $1$ we will prove that if we change $\underline{\mu}'$ we get linearly dependent gradients and in Claim $2$ we will show that, for a fixed $\underline{\mu}'$, $ \nabla F_{(\underline{\mu}',L+1,L+2)}(\mathbf{X}), \dots, \nabla F_{(\underline{\mu}',L+1,L+1+h_n)}(\mathbf{X}) $ are linearly independent and all the gradients of the generators of  $ I(\mathcal{X}_n) $ depend linearly on them.
\medskip

\noindent{\bf Claim 1}: The gradients of $ F_{(\underline{\mu_1}', L+1+i_1, L+1+i_2)} $ and $ F_{(\underline{\mu_2}', L+1+i_1, L+1+i_2)} $ are linearly dependent for every $ \underline{\mu_1}' $ and $ \underline{\mu_2}'$ with length $ k+n-1$, in $[1, L]$, and every $ i_1 < i_2 $ in the range $ [0, h_n]$.

Indeed, for the sub--matrix of $ M_{\mathcal{X}_n}(\mathbf{X}) $ obtained by canceling the last column, we can construct a Pl\"ucker relation from the multi--indices $ \underline{\mu_1}'$ and $ (\underline{\mu_2}',\ $ $ L+1+i_1, L+1+i_2)$. This relation has a priori three summands but one of them vanishes. Hence it reduces to
\begin{equation*} \begin{split}
(-1)^{k+n} & D_{(\underline{\mu_1}', L+1+i_1)}(\mathbf{X}) D_{(\underline{\mu_2}', L+1+i_2)}(\mathbf{X}) + \\ & (-1)^{k+n+1} D_{(\underline{\mu_1}', L+1+i_2)}(\mathbf{X}) D_{(\underline{\mu_2}', L+1+i_1)}(\mathbf{X}) = 0.
\end{split} \end{equation*}
This relation is equivalent to saying that the gradients of $ F_{(\underline{\mu_1}', L+1+i_1, L+1+i_2)} $ and $ F_{(\underline{\mu_2}', L+1+i_1, L+1+i_2)} $ are linearly dependent. \medskip

\noindent{\bf Claim $2$}: $ \nabla F_{(\underline{\mu}',L+1,L+2)}(\mathbf{X}), \dots, \nabla F_{(\underline{\mu}',L+1,L+1+h_n)}(\mathbf{X}) $ are a basis of the vector space spanned by the gradients of the generators of $ I(\mathcal{X}_n) $ evaluated at $ \mathbf{X}$, for a fixed $\underline{\mu}'.$

For $\underline{\lambda}'$ as in $(5)$,  $\nabla F_{(\underline{\lambda}',L+1,L+2)}(\mathbf{X}), \dots, \nabla F_{(\underline{\lambda}',L+1,L+1+h_n)}(\mathbf{X}) $ are linearly independent as $ D_{(\underline{\underline{\lambda}'},L+1)}(\mathbf{X}) \not= 0$.

From the previous claim and the following identity, that is an immediate consequence of (\ref{gradient}),
\[
\begin{split}
D_{(\underline{\lambda}',L+1)}(\mathbf{X}) \nabla F_{(\underline{\lambda}',L+1+i,L+1+j)}(\mathbf{X}) -& D_{(\underline{\lambda}',L+1+i)}(\mathbf{X}) \nabla F_{(\underline{\lambda}',L+1,L+1+j)}(\mathbf{X}) +\\ + D_{(\underline{\lambda}',L+1+j)}(\mathbf{X}) & \nabla F_{(\underline{\lambda}',L+1,L+1+i)}(\mathbf{X}) = \vec{0},
\end{split}
\]
we get that all gradients of generators of $ I(\mathcal{X}_n) $ associated to multi--indices $ (\underline{\lambda}', L+1+i_1, L+1+i_2) $ are in the span of the $ h_n $ gradients above. To complete the proof of the claim, we prove that $ \nabla F_{(\underline{\mu}', L+1+i_1, \dots, L+1+i_t)}(\mathbf{X}) $ linearly depends on the gradients of $ F_{(\underline{\tau}', L+1+i_1, L+1+i_2)}, \dots, F_{(\underline{\tau}', L+1+i_1, L+1+i_t)} $ evaluated at $ \mathbf{X}$, for a suitable $ \underline{\tau'}$. To avoid trivial cases, we assume that $ \nabla F_{(\underline{\mu}', L+1+i_1, \dots, L+1+i_t)}(\mathbf{X}) \not= \vec{0}$. Then, $ D_{(\underline{\mu}', L+1+i_1, \dots, L+1+i_t) \setminus{L+1+i_j}}(\mathbf{X}) \not=0$ for a suitable $ j$. To fix ideas, we assume $ j=t$. Since the rank of $ M_{\mathcal{X}_{n-1}}(\mathbf{X}) $ is equal to $ k+n-1$, whilst the one of $ M_{\mathcal{X}_n}(\mathbf{X}) $ without its last column is equal to $ k+n$, there exist $t-2$ suitable rows in $ M_{\mathcal{X}_{n-1}}(\mathbf{X}) $ with which to replace $ t-2 $ rows among $ L+1+i_1, \dots, L+1+i_{t-1}$, so to get $ D_{(\underline{\tau'},L+1+i_h)}(\mathbf{X}) \not= 0$. To fix ideas, we set $ h=1$.

The equality
\[
\begin{split}
\sum_{j=2}^t (-1)^{t-j+1} & D_{(\underline{\mu}', L+1+i_1, \dots, L+1+i_t)\setminus\{L+1+i_j\}}(\mathbf{X}) \nabla F_{(\underline{\tau}',L+1+i_1,L+1+i_j)}(\mathbf{X}) + \\ + & D_{(\underline{\tau}',L+1+i_1)}(\mathbf{X}) \nabla F_{(\underline{\mu}', L+1+i_1, \dots, L+1+i_t)}(\mathbf{X}) = \vec{0},
\end{split}
\]
shows that $ \nabla F_{(\underline{\mu}', L+1+i_1, \dots, L+1+i_t)}(\mathbf{X}) $ is a linear combination of $$ \nabla F_{(\underline{\tau}',L+1,L+2)}(\mathbf{X}), \dots, \nabla F_{(\underline{\tau}',L+1,L+1+h_n)}(\mathbf{X}) $$ for every length $ k+n+1-t $ multi--index $ \underline{\mu}'$, and every $ t \geq 2$. Now, we prove the above equality. It is easy to check that the coefficient of $ \vec{e}_{i_1} $ is equal to
\[
\begin{split}
(-1)^{t-1} D_{(\underline{\tau}',L+1+i_1)}(\mathbf{X})& D_{(\underline{\mu}',L+1+i_2, \dots, L+1+i_t)}(\mathbf{X}) -\\ - \sum_{j=2}^t (-1)^{t-j+1} D_{(\underline{\tau}',L+1+i_j)}(\mathbf{X}) & D_{(\underline{\mu}',L+1+i_1, \dots, L+1+i_t)\setminus\{L+1+i_j\}}(\mathbf{X})
\end{split}
\]
while the coefficients of $ \vec{e}_{i_2}, \dots, \vec{e}_{i_t} $ vanish. The coefficient of $ \vec{e}_{i_1} $ is the Pl\"ucker relation obtained from the multi--indices $ \underline{\tau}' $ and $ (\underline{\mu}', L+1+i_1, \dots, L+1+i_t) $ and so it vanishes as well, and the proof of the claim is complete.

It follows that the rank of the jacobian matrix of $ I(\mathcal{X}_n) $ at $ \mathbf{X} $ is equal to $ h_n$.
\end{proof}

To complete the analysis, we describe the critical locus $ \mathcal{X}_n $ at $ \mathbf{X}$.
\begin{corollary} \label{corol-nest2}
In the same assumptions as Proposition \ref{prop-nest2}, we have that the dimension of $ \mathcal{X}_n $ at $ \mathbf{X} $ is at most $ k - h_n$. Moreover, if $ \dim \mathcal{X}_n = k-h_n-1$ at $ \mathbf{X}$, then $ \mathbf{X} $ is singular for $ \mathcal{X}_n$. Lastly, if $ C_{Q_n} \subseteq \mathcal{X}_n $ and $ \dim \mathcal{X}_n = k-h_n-1$, then $ C_{Q_n} $ is an irreducible component of $ \mathcal{X}_n $ with multiplicity $ 2$.
\end{corollary}


\section{Study of the maps $ \pi_1, \pi_2, \pi_2 \circ \pi_1^{-1}, \pi_1 \circ \pi_2^{-1}$}
\label{maps-sect}

In this section, we consider the commutative diagram (\ref{main-cd}), and study the projections $ \pi_1, \pi_2$, and the composite maps $ \pi_2 \circ \pi_1^{-1}, \pi_1 \circ \pi_2^{-1}$. Some natural guesses are that $ \pi_i, i=1,2$, is surjective or at least dominant, and the two last ones are birational. However, Example \ref{example-3-1} shows that these naive guesses do not always hold.

For explanatory purposes, because of the symmetry between $ \pi_1: \mathcal{U} \to \mathcal{X} $ and $ \pi_2: \mathcal{U} \to \mathcal{Y}$, we address map $ \pi_1$ in the case of one projection and map $\pi_2 \circ \pi_1^{-1}$ for the relationship between $ \mathcal{X}$ and $ \mathcal{Y}$.

At first, we characterize the points of the image $ \pi_1(\mathcal{U}) \subseteq \mathcal{X}$ in the most general situation (see Theorem \ref{fibra}), and then we adapt it to the case of interest in Multiview Geometry, namely when the projection centers do not intersect (see Corollary \ref{cor-4-2}).
\begin{theorem} \label{fibra}
Let $ \mathcal{X} \subset \mathbb{P}^k $ be the critical locus associated to the two $n$--tuples of projections $ Q_1, \dots, Q_n $ and $ P_1, \dots, P_n$. Let $ \mathbf{X} \in \mathcal{X}$. $ \mathbf{X} \in \pi_1(\mathcal{U}) $ if and only if either $ \mathbf{X} \notin C_1 \cup \dots \cup C_n$, or $ \mathbf{X} \in C_{r+1} \cap \dots \cap C_n \cap (\mathcal{X}_r \setminus \{ C_1 \cup \dots \cup C_r \}) $ for some $ 1 \leq r < n$, up to renumbering the projections, where $ \mathcal{X}_r $ is the critical locus for the associated conjugate projections $ P_1, \dots, P_r$ and $ Q_1, \dots, Q_r$.
\end{theorem}

\begin{proof} Let us assume first that $ \mathbf{X} \notin C_1 \cup \dots \cup C_n$, for $ \mathbf{X} \in \mathcal{X}$.

As a current notation, we set $M_{\mathcal{X}}(\mathbf{X})$ the matrix $M_{\mathcal{X}}$ evaluated at $ \mathbf{X}$. The rank of
$ M_{\mathcal{X}}(\mathbf{X}) $ is smaller than $ k+n+1 $ and so the homogeneous linear system $ M_{\mathcal{X}}(\mathbf{X}) \overline{Y} = 0 $ has non--trivial solutions, where $ \overline{Y} = (Y, v_1, \dots, v_n)^T$. We can rewrite the previous system as $$ P_i(Y) + v_i Q_i(\mathbf{X}) = 0, \qquad i=1, \dots, n.$$ We remark that $ Y \not= 0$. Indeed, by contradiction, if $ Y = 0$, then $ v_i \ Q_i(\mathbf{X}) = 0$, too, and so, for $i=1, \dots, n$, $ v_i = 0 $ because $ \mathbf{X} \notin C_1 \cup \dots \cup C_n$, and this is not possible, since $ \overline{Y} $ is non--trivial. Hence, every such $ Y $ gives us a projective point $ \mathbf{Y}$. If $ \mathbf{Y} $ is in the center of projection $ P_j $ for some $ j=1, \dots, n$, then it belongs to the critical locus $ \mathcal{Y}$, as previously remarked. Otherwise, the homogeneous linear system $ M_{\mathcal{Y}}(\mathbf{Y}) \overline{X} = 0 $ has $ (X, 1/v_1, \dots, 1/v_n)^T $ as solution, and so the rank of $ M_{\mathcal{Y}}(\mathbf{Y}) $ is at most $k+n$, and hence $ \mathbf{Y} $ is in the critical locus $ \mathcal{Y}$. Therefore, the couple $ (\mathbf{X}, \mathbf{Y}) $ is in the unified critical locus $ \mathcal{U}$, and $\mathbf{X} \in \pi_1(\mathcal{U})$.

Now, we assume that $ \mathbf{X} \in C_{r+1} \cap \dots \cap C_n \cap (\mathcal{X}_r \setminus \{ C_1 \cup \dots \cup C_r \}) $ for some $ 1 \leq r < n$, up to renumbering the projections.

Let $ M_{\mathcal{X}_r}(\mathbf{X}) $ be the submatrix of $ M_{\mathcal{X}}(\mathbf{X}) $ obtained by delating the rows corresponding to projections $ P_{r+1}, \dots, P_n $ and the last $ n-r $ columns. Since $ \mathbf{X} \in \mathcal{X}_r$, the homogeneous linear system $ M_{\mathcal{X}_r}(\mathbf{X}) \overline{Y} = 0 $ has non--trivial solutions, where $ \overline{Y} = (Y, v_1, \dots, v_r)^T$. As in the first part of the proof, in a non--zero solution, $ Y \not= 0$.

To prove that the projective point $ \mathbf{Y} $ corresponding to $ Y $ verifies $ (\mathbf{X}, \mathbf{Y}) $ $ \in \mathcal{U}$, one can argue as in the first part of the proof. The main difference is that, in this case, a solution of $ M_{\mathcal{Y}}(\mathbf{Y}) \overline{X} = 0 $ is $ (X, 1/v_1, \dots, 1/v_r, 0, \dots, 0)^T $ and hence $ \mathbf{X} \in \pi_1(\mathcal{U})$ also in this case.

To prove the converse, because of Proposition \ref{prop-nest1}, it is enough to assume that $ \mathbf{X} \in C_{r+1} \cap \dots \cap C_n$, with $ \mathbf{X} \notin \mathcal{X}_r$, for some $ 1 \leq r < n$, up to renumbering the projections, and to get a contradiction.

Let us assume that there exists $ \mathbf{Y} $ such that $ (\mathbf{X}, \mathbf{Y}) \in \mathcal{U}$ so that $ \mathbf{X} \in \pi_1(\mathcal{U})$. Since $ \mathbf{X} \notin \mathcal{X}_r$, then $ \mathbf{X} \notin C_1 \cup \dots \cup C_r$, and so $ Q_i(\mathbf{X}) \not= 0 $ for $ 1 \leq i \leq r$. Furthermore, by definition of unified critical locus, $ P_i(\mathbf{Y}) = Q_i(\mathbf{X}) $ for $ i=1, \dots, n$. Hence, the homogeneous linear system $ M_{\mathcal{X}_r}(\mathbf{X}) \overline{Y} = 0 $ has a non--trivial solution and so both the rank of $ M_{\mathcal{X}_r}(\mathbf{X}) $ is less than $ k+r+1 $ and $ \mathbf{X} \in \mathcal{X}_r$, that is not possible, because of our assumption.
\end{proof}

In Examples \ref{example-P3P2P2} and \ref{example-3-1} in next section, we illustrate both cases in Theorem \ref{fibra}. Indeed, for $ r=2$, in the list of the main properties of $ \pi_1 $ and $ \pi_2$, cases $(5), (6)$ correspond to points not in the union of the centers, and, for $ r=3$, the cases $(a), (b), (c), (d)$ in the list for $ \tilde{\pi}_1, \tilde{\pi}_2$, correspond to points either not in the union of the centers, or in a center and in the critical locus for two projections over three.

\begin{corollary} \label{cor-4-2}
Assume $ C_i \cap C_j = \emptyset $ for $ 1 \leq i < j \leq n$. Let $ \mathbf{X} \in \mathcal{X}$, and, up to renumbering the projections, let $ \mathcal{X}_{n-1} $ be the critical locus for the associated conjugate projections $ P_1, \dots, P_{n-1}$ and $ Q_1, \dots, Q_{n-1}$. Then, $ \mathbf{X} \in \pi_1(\mathcal{U}) $ if and only if either $ \mathbf{X} \notin C_1 \cup \dots \cup C_n$, or $ \mathbf{X} \in C_n \cap \mathcal{X}_{n-1}$.
\end{corollary}

We can more precisely describe the fibre over $ \mathbf{X} \in \pi_1(\mathcal{U}) \subseteq \mathcal{X}$, under suitable assumptions on $ \mathbf{X}$.
\begin{proposition} \label{fibra-1}
Let $ \mathbf{X} \in \mathcal{X}$. If $ \mathbf{X} $ is smooth for $ \mathcal{X} $ and $ \mathbf{X} \notin C_1 \cup \dots \cup C_n$, then $ \vert \pi_1^{-1}(\mathbf{X}) \vert = 1$.
\end{proposition}

\begin{proof} From Corollary \ref{cor-4-2}, it follows that $ \mathbf{X} \in \pi_1(\mathcal{U})$. Moreover, $ M_{\mathcal{X}}(\mathbf{X}) $ has rank $ k+n $ so that the homogeneous linear system $ M_{\mathcal{X}}(\mathbf{X}) \overline{Y} = 0 $ has one and only one non--trivial solution, up to a scalar, and the claim follows.
\end{proof}

\begin{remark} \rm
The true hypothesis we need in Proposition \ref{fibra-1} is that the rank of $ M_{\mathcal{X}}(\mathbf{X}) $ is equal to $ k+n$. This happens surely if $ \mathbf{X} $ is smooth. When $ \mathbf{X} $ is singular, it can happen that $ M_{\mathcal{X}}(\mathbf{X}) $ has rank $ k+n$, or has smaller rank. E.g., in Example \ref{example-3-1}, the point $ A $ is singular for $ \mathcal{X}$, nevertheless $ M_{\mathcal{X}}(A) $ has rank $ k+n=6$. In next Example \ref{example-3-2}, we'll consider the case $ \mbox{rank}(M_{\mathcal{X}}(\mathbf{X})) < k+n$.
\end{remark}

The remaining part of the section is devoted to the study of the composite map $ \pi_2 \circ \pi_1^{-1}$. We first consider the case of (points in) a projection center, and we have to distinguish the cases of $ n=2 $ projections from the case of $ n \geq 3 $ ones. Theorem \ref{birational}, the last result of the section, concerns the birationality of the critical loci.

With the same technique of the proof of Proposition \ref{fibra-1}, we get the fibre of $ \pi_1 $ over a point in a projection center in the case $ n=2$.
\begin{proposition} \label{fibra-2}
Let $ P_1, P_2 $ and $ Q_1, Q_2 $ be a couple of two projections. Let $ C_1 $ be the center of $ Q_1$, and let $ \mathbf{X} \in C_1$. Then,
$ \pi_2 \circ \pi_1^{-1}(\mathbf{X}) $ is a linear space of dimension $ k-h_2$ containing the center of $ P_2$, while $ \pi_2 \circ \pi_1^{-1}(C_1) $ is a linear space of dimension at most $ 2k-h_1-h_2-1$.
\end{proposition}

\begin{proof} The critical locus of the only second projection is $ \mathbb{P}^k$, and so $ \mathbf{X} $ satisfies the hypotheses of Theorem \ref{fibra} with $ r=1, n=2$. As explained in the proof of the same Theorem, every $ \mathbf{Y} $ such that $ \overline{Y} = (Y, v_2)^T $ is a solution of $$ \left( \begin{array}{cc} P_2 & Q_2(\mathbf{X}) \end{array} \right) \overline{Y} = 0 $$ is the second element of a couple in $ \pi_1^{-1}(\mathbf{X})$. Since the solutions fill a linear space, the fiber is a linear space, as claimed. Its dimension follows from the assumptions on the rank of $ P_2$.

The linear space $ C_1 $ is the linear span of the points $ \mathbf{A}_1, \dots, \mathbf{A}_{k-h_1}$. When we solve the linear system
\[
\left( \begin{array}{cc} P_2 & Q_2(\mathbf{A_j}) \end{array} \right) \overline{Y} = 0, \quad \mbox{ for } j=1, \dots, k-h_2,
\]
its solutions can be either $ (Y, 0)^T$, where the associated point $ \mathbf{Y} $ is in the center $ C_{P_2}$ of $ P_2$, or $ (B_j, -1)^T$, for a suitable point $ \mathbf{B}_j$. Let $ L $ be the linear span of $ C_{P_2} $ and $ \mathbf{B}_1, \dots, \mathbf{B}_{k-h_1}$, and let $ \mathbf{Y} \in L$. Then, there is a point $ \mathbf{B} \in C_{P_2}$, and $ t_0, t_1, \dots, t_{k-h_1} \in \mathbb{C}$, not all equal to zero, such that $ \mathbf{Y} = t_0 \mathbf{B} + t_1 \mathbf{B}_1 + \dots + t_{k-h_1} \mathbf{B}_{k-h_1}$. The linear system $ (P_2 \ Q_2(t_1 \mathbf{A}_1 + \dots + t_{k-h_1} \mathbf{A}_{k-h_1})) \overline{Y} = 0 $ has $ (Y, -1)^T $ among its solutions, and so $ \mathbf{Y} \in \pi_2 \circ \pi_1^{-1}(C_1)$. The inverse inclusion is analogous, and so we get that $ L = \pi_2 \circ \pi_1^{-1}(C_1)$. The linear space $ L $ is spanned by $ 2k - h_1-h_2 $ points and so its dimension is not larger than $ 2k-h_1-h_2-1$, as claimed.
\end{proof}

\begin{remark} \rm
In \cite{bnt2}, authors classify and study smooth critical loci. In the case of two views, every minimal degree variety, but the Veronese surface, is critical for suitable couples of two projections from $ \mathbb{P}^k $ to $ \mathbb{P}^{h_1}, \mathbb{P}^{h_2}$, under the assumption that $ h_1+h_2 \geq k+1$, and some generality hypotheses on the projections. Even if the critical locus is not smooth, but of the expected dimension, it is defined by the vanishing of the order $ 2 $ minors of a $ (h_1+h_2+1-k) \times 2 $ matrix of linear forms. Hence, the critical locus contains linear spaces of dimension $ k-(h_1+h_2+1-k) = 2k-h_1-h_2-1$, that agrees with the result in Proposition \ref{fibra-2} on the dimension of $ \pi_2 \circ \pi_1^{-1}(C_1)$.
\end{remark}

The inverse image of a center is no more a linear space when the number of projections increases.
\begin{proposition} \label{fibra-3}
Let $ P_1, \dots, P_n $ and $ Q_1, \dots, Q_n $ be a couple of $n \geq 3$ projections. Up to renumbering the projections, let $ C_n $ be the center of $ Q_n$, and let us assume that $ C_n $ is contained in the smooth locus of $ \mathcal{X}_{n-1}$, critical locus of the first $ n-1 $ projections. Assume moreover that $ \sum_{i \not= n} h_i \geq k$. Then,
$ \pi_2 \circ \pi_1^{-1}(C_n) $ is (a projection of) the $n-1$--tuple Veronese embedding of $ C_n $ in $ \mathcal{X}$.
\end{proposition}

\begin{proof} If $ \mathbf{X} \in C_n$, then $ \mathbf{X} $ is smooth in $ \mathcal{X}_{n-1}$. From Proposition \ref{fibra-1} and Theorem \ref{fibra}, it follows that there is only a point $ \mathbf{Y} \in \mathcal{Y} $ such that $ (\mathbf{X}, \mathbf{Y}) \in \mathcal{U}$. To compute such a point, it is enough to solve the system $ M_{\mathcal{X}_{n-1}}(\mathbf{X}) \overline{Y} = 0$. Since it has a unique solution, we can compute it by taking the maximal minors of a subset of $ k+n-1 $ independent equations from the $ \sum_{i \not= n} h_i + n -1 $ given ones. To get independent equations, we have to take at least an equation from every projection. The maximal minors that give the coordinate of $ \mathbf{Y} $ are then homogeneous polynomials of degree $ n-1 $ in the $ k-h_n $ variables that parameterize the points in $ C_n$, and so the claim follows.
\end{proof}

To illustrate the result above, one can consider the case of three general projections from $ \mathbb{P}^4 $ to $ \mathbb{P}^2$. In such a case, both $ \mathcal{X} $ and $ \mathcal{Y} $ are Bordiga surfaces (see \cite{bnt1}). The critical locus associated to two of the three projections is $ \mathbb{P}^4$ and $ h_1 + h_2 = k$. The center of any of the three projections $ Q_1, Q_2, Q_3$ is a line, say $ L$. Then, the assumptions of Proposition \ref{fibra-3} are fulfilled, and so $ \pi_2 \circ \pi_1^{-1}(L) $ is a conic into the Bordiga surface $ \mathcal{Y}$. \bigskip

The cases when $ \sum_{i \not= n} h_i < k$, could be handled analogously to Propositions \ref{fibra-2} and \ref{fibra-3}, but the results can hardly be collected in a unique general statement. \bigskip

We conclude this section with a result that relates the geometry of the associated critical loci $ \mathcal{X} $ and $ \mathcal{Y}$, and answers to what extent they are birational varieties.
\begin{theorem} \label{birational}
Let $ P_1, \dots, P_n $ and $ Q_1, \dots, Q_n $ be a couple of $ n $ projections from $ \mathbb{P}^k $ to $ \mathbb{P}^{h_1}, \dots, \mathbb{P}^{h_n}$. Let $ \mathcal{X} $ and $ \mathcal{Y} $ be the associated critical loci. Let $ (\mathbf{X_0}, \mathbf{Y_0}) \in \mathcal{U}$ be a point such that $ \mathbf{X_0} $ is smooth on an irreducible component $ \mathcal{X}'$ of $ \mathcal{X}$, and $ \mathbf{Y_0} $ is smooth on an irreducible component $ \mathcal{Y}'$ of $ \mathcal{Y}$. Then $ \pi_2 \circ \pi_1^{-1}: \mathcal{X}' \dashrightarrow \mathcal{Y}'$ is a birational map.
\end{theorem}

\begin{proof} Since $ M_{\mathcal{X}}(\mathbf{X_0}) $ has rank $ k+n$, there exists an open subset $ \mathcal{A} $ of $ \mathcal{X}' $ such that $ M_{\mathcal{X}}(\mathbf{X}) $ has rank $ k+n $ at every point $ \mathbf{X} \in \mathcal{A}$. Because of Proposition \ref{fibra-1}, for every point $ \mathbf{X} \in \mathcal{A} $ there exists a unique point $ \mathbf{Y} \in \mathcal{Y}$ such that $ (\mathbf{X}, \mathbf{Y}) \in \mathcal{U}$, unified critical locus. The locus $ \{ \mathbf{Y} \in \mathcal{Y} \ \vert \ (\mathbf{X}, \mathbf{Y}) \in \mathcal{U}, \mbox{ for some } \mathbf{X} \in \mathcal{A} \} $ is irreducible with the same dimension as $ \mathcal{Y}'$, irreducible component containing $ \mathbf{Y_0} $ because it contains $ \mathbf{Y_0} $ that is smooth. Then, the locus above is contained in the irreducible component $ \mathcal{Y}'$ of $ \mathcal{Y}$. Then, $ \pi_2 \circ \pi_1^{-1}: \mathcal{X}' \dashrightarrow \mathcal{Y}' $ is a birational map, as claimed.
\end{proof}

As the examples in next section show, Theorem \ref{birational} is the best result we can obtain in comparing the geometry of the associated critical loci $ \mathcal{X} $ and $ \mathcal{Y}$.

\begin{remark} \rm
Even in the case when both $ \mathcal{X} $ and $ \mathcal{Y} $ are smooth (see \cite{bnt2}), the map $ \pi_2 \circ \pi_1^{-1} $ is not necessarily an isomorphism. Indeed, in such a case, the centers of projections are contained in them, and Propositions \ref{fibra-2}, \ref{fibra-3} show that their images either are not of the right dimensions, or are not linear spaces.
\end{remark}


\section{Examples} \label{sec::examples}

The first example we produce concerns the classical setting of two pairs of projections from $ \mathbb{P}^3 $ to $ \mathbb{P}^2$ and illustrates the construction of the critical loci.
\begin{example} \label{example-P3P2P2}\rm
Let us choose reference frames in $ \mathbb{P}^3$, and in two views equal to $ \mathbb{P}^2$, and let us consider the projections $ P_1, P_2, Q_1, Q_2 $ from $ \mathbb{P}^3 $ to $ \mathbb{P}^2$ whose matrices are
$$ P_1 = \left( \begin{array}{cccc} 1 & 0 & 0 & 0 \\ 0 & 1 & 0 & 0 \\ 0 & 0 & 1 & 0 \end{array} \right) \qquad Q_1 = \left( \begin{array}{cccc} 1 & 0 & 0 & 1 \\ 1 & 0 & 1 & -1 \\ 1 & 1 & -1 & 0 \end{array} \right),$$ $$ P_2 = \left( \begin{array}{cccc} 0 & 1 & 0 & 1 \\ 1 & 0 & 1 & 1 \\ 1 & 1 & 0 & 1 \end{array} \right) \qquad Q_2 = \left( \begin{array}{cccc} 1 & -1 & 0 & 0 \\ 0 & 1 & 1 & 0 \\ 0 & 0 & 1 & -1 \end{array} \right).$$ The critical loci $ \mathcal{X} $ and $ \mathcal{Y} $ are quadrics, vanishing loci of $ \det(M_{\mathcal{X}}), \det(M_{\mathcal{Y}})$, respectively. Since $$ M_{\mathcal{X}} = \left( \begin{array}{cccccc} 1 & 0 & 0 & 0 & x_0+x_3 & 0 \\ 0 & 1 & 0 & 0 & x_0+x_2-x_3 & 0 \\ 0 & 0 & 1 & 0 & x_0+x_1-x_2 & 0 \\ 0 & 1 & 0 & 1 & 0 & x_0-x_1 \\ 1 & 0 & 1 & 1 & 0 & x_1+x_2 \\ 1 & 1 & 0 & 1 & 0 & x_2-x_3 \end{array}\right) $$ and $$  M_{\mathcal{Y}} = \left( \begin{array}{cccccc} 1 & 0 & 0 & 1 & y_0 & 0 \\ 1 & 0 & 1 & -1 & y_1 & 0 \\ 1 & 1 & -1 & 0 & y_2 & 0 \\ 1 & -1 & 0 & 0 & 0 & y_1+y_3 \\ 0 & 1 & 1 & 0 & 0 & y_0+y_2+y_3 \\ 0 & 0 & 1 & -1 & 0 & y_0+y_1+y_3 \end{array} \right),$$ we have $$ \mathcal{X}: 2 x_0x_1 - x_1^2 -2 x_0x_2 + x_1x_2 + 2 x_2^2 +2 x_0x_3 + x_1x_3 -3 x_2x_3 +2 x_3^2 = 0 $$ and $$ \mathcal{Y}: -4 y_0^2 -5 y_0y_1 - y_0y_2 +3 y_1y_2 - y_2^2 -6 y_0y_3 +3 y_1y_3 - y_2y_3 = 0.$$ The unified critical locus $ \mathcal{U} $ is defined by the vanishing of the determinants of $ M_{\mathcal{X}} $ and $ M_{\mathcal{Y}}$, and by the vanishing of the $2 \times 2$ minors of the matrices $$ \left( \begin{array}{cc} y_0 & x_0+x_3 \\ y_1 & x_0+x_2-x_3 \\ y_2 & x_0+x_1-x_2 \end{array} \right), \quad \left( \begin{array}{cc} y_1+y_3 & x_0-x_1 \\ y_0+y_2+y_3 & x_1+x_2 \\ y_0+y_1+y_3 & x_2-x_3 \end{array} \right),$$ and so its ideal is generated by the two quadrics above and by further $6$ bi--homogeneous forms, linear in both $ x_0, \dots, x_3$, and $ y_0, \dots, y_3$.

Now, we describe the geometry of the maps $ \pi_1: \mathcal{U} \to \mathcal{X} $ and $ \pi_2: \mathcal{U} \to \mathcal{Y}$, where $ \mathcal{X} $ and $ \mathcal{Y} $ are the critical quadrics computed before, and $ \mathcal{U} $ is their unified critical locus.

The centers of $ P_1 $ and $ P_2 $ are $ C_{P_1}(0:0:0:1) $ and $ C_{P_2}(0:-1:-1:1)$, respectively, while $ C_{Q_1}(-1:3:2:1) $ and $ C_{Q_2}(-1:-1:1:1)$ are the centers of $ Q_1$ and $ Q_2$, respectively. By Proposition \ref{prop-2-1}, $ C_{P_1}, C_{P_2} \in \mathcal{Y}$, and $ C_{Q_1}, C_{Q_2} \in \mathcal{X}$. Furthermore, it is easy to check that the critical quadrics $ \mathcal{X}, \mathcal{Y} $ are smooth.

Now, we list the main properties of $ \pi_1, \pi_2$.
\begin{enumerate}
\item $ \pi_1^{-1}(C_{Q_1}) = \{(C_{Q_1}, P) \in \mathcal{U} \ \vert \ P \in r_1: 4y_0+5y_1+5y_3 = y_2+y_3 = 0 \}$;
\item $ \pi_1^{-1}(C_{Q_2}) = \{(C_{Q_2}, P) \in \mathcal{U} \ \vert \ P \in r_2: y_0 = 3y_1 - y_2 = 0 \}$;
\item $ \pi_2^{-1}(C_{P_1}) = \{(Q, C_{P_1}) \in \mathcal{U} \ \vert \ Q \in s_1: x_1+x_3 = x_0-x_2+2x_3 = 0 \}$;
\item $ \pi_2^{-1}(C_{P_2}) = \{(Q, C_{P_2}) \in \mathcal{U} \ \vert \ Q \in s_2: x_1-2x_2+x_3 = x_0+x_3 = 0 \}$;
\item $ \pi_1^{-1}(Q) $ is a point in $ \mathcal{U} $ for every $ Q \in \mathcal{X} \setminus \{C_{Q_1}, C_{Q_2}\}$;
\item $ \pi_2^{-1}(P) $ is a point in $ \mathcal{U} $ for every $ P \in \mathcal{Y} \setminus \{C_{P_1}, C_{P_2}\}$.
\end{enumerate}
Moreover, $ C_{P_2} \in r_1, C_{P_1} \in r_2, C_{Q_2} \in s_1, C_{Q_1} \in s_2$.

As this Example shows, $ \pi_1 $ and $ \pi_2 $ are surjective, and $ \pi_2 \circ \pi_1^{-1} $ is a birational map, according also to \cite{bratelund}.
\end{example}

The second example illustrates a case of nested critical loci, and shows that it does not hold either the naive guess that $ \pi_1(\mathcal{U}) = \mathcal{X} $ or the naive guess that $ \mathcal{X} $ and $ \mathcal{Y} $ are birational.

\begin{example} \label{example-3-1} \rm
We consider a couple of three projections from $ \mathbb{P}^3$, the first two to $ \mathbb{P}^2$ described in Example \ref{example-P3P2P2}, and the third to $ \mathbb{P}^1$. We adopt the same notation in section \ref{sec::nestingcritical loci}, and so the quadrics computed above are now $ \mathcal{X}_2 $ and $ \mathcal{Y}_2$.

To complete the picture of Example \ref{example-P3P2P2}, we set $ r'_1, r'_2 $ the lines in $\mathcal{Y}_2$ through $ C_{P_2}, C_{P_1}$, different from $ r_1, r_2$, respectively, and $ s'_1, s'_2 $ the lines in $ \mathcal{X}_2$, through $ C_{Q_2}, C_{Q_1}$, different from $ s_1, s_2$, respectively.

Now, we consider also the third projection. The center $ L_{P_3}$ of the third projection $ P_3 $ is a line, and analogously $ L_{Q_3}$, center of $ Q_3$. The critical locus will have different properties according to the positions of $ L_{P_3} $ with respect to $ \mathcal{Y}_2$, and of $ L_{Q_3} $ with respect to $ \mathcal{X}_2$. Different situations occur not only when the centers are contained in the quadrics, but, if they are not contained, also when they meet one or two of the four lines $r_1, r'_1, r_2, r'_2$ or $ s_1, s'_1, s_2, s'_2$ in the quadrics through the two centers of the above projections.

We completely describe the geometry of the critical loci for the case
$$ P_3 = \left( \begin{array}{cccc} 0 & 3 & 0 & 1 \\ 0 & 0 & 1 & 1 \end{array} \right) \qquad \mbox{ and } \qquad Q_3 = \left( \begin{array}{cccc} 0 & 0 & 0 & 1 \\ 0 & 1 & 0 & 0 \end{array} \right).$$

Now, we compute the critical loci $ \mathcal{X}_3, \mathcal{Y}_3$, the unified one $ \tilde{\mathcal{U}}$, and study the maps $ \tilde{\pi}_1: \tilde{\mathcal{U}} \to \mathcal{X}_3$ and $ \tilde{\pi}_2:\tilde{\mathcal{U}} \to \mathcal{Y}_3$, for the couple of three projections $ P_1, P_2, P_3$, and $ Q_1, Q_2, Q_3$.

As previously stated, $ \mathcal{X}_3 $ and $ \mathcal{Y}_3 $ are determinantal (see Proposition \ref{prop-2-1}), and, in the case under consideration, have dimension $ 1 $ and degree $ 6$. Since the projection centers are contained in the critical locus, and the centers of $ P_3 $ and $ Q_3$ are lines, neither $ \mathcal{X}_3$, nor $ \mathcal{Y}_3 $ is irreducible. The centers of $ Q_3 $ and $ P_3 $ are $ L_{Q_3}: x_1 = x_3 = 0$ and $ L_{P_3}: 3 y_1 + y_3 = y_2 + y_3 = 0$, respectively. Also the residual curves $ \tilde{\mathcal{X}} $ and $ \tilde{\mathcal{Y}} $ are determinantal, respectively defined by the vanishing of the maximal minors of $$ N_{\tilde{\mathcal{X}}} = \left( \begin{array}{cc}
\begin{array}{l} 2 x_0 x_2 + x_1 x_2 -3 x_2^2 +\\ + x_2 x_3 +4 x_3^2 \end{array} & \begin{array}{l} x_1^2 +7 x_1 x_2 - x_2^2 -\\ -11 x_1 x_3 -5 x_2 x_3 \end{array} \\
-x_1 + x_2 - x_3 & x_2 \\ -2 x_2 + x_3 & -2 x_0 + x_1 -2 x_3 \end{array} \right) $$ and $$ N_{\tilde{\mathcal{Y}}} = \left( \begin{array}{cc}
\begin{array}{l} 24 y_1 y_2 -3 y_1 y_3 +\\ +3 y_2 y_3 -6 y_3^2\end{array} & \begin{array}{l} 18 y_1^2 +6 y_2^2 +15 y_1 y_3 +\\ +9 y_2 y_3 +6 y_3^2\end{array} \\  \\
6 y_0 -2 y_2 -2 y_3 & 2 y_0 +2 y_2 +2 y_3 \\ -6 y_1 +8 y_2 +11 y_3 & -8 y_0 -12 y_1 -2 y_2 -11 y_3 \end{array} \right).$$

\noindent $ \tilde{\mathcal{X}}$ and $ \tilde{\mathcal{Y}}$ are ACM, irreducible, degree $ 5$ curves, contained in the critical quadrics $ \mathcal{X}_2 $ and $ \mathcal{Y}_2$, respectively. $ \tilde{\mathcal{X}} $ meets $ L_{Q_3} $ at $ A(1:0:0:0) $ and $ B(1:0:1:0)$, while $ \tilde{\mathcal{Y}} $ meets $ L_{P_3} $ at $ C(0:1:3:-3) $ and $ D(5:2:6:-6)$. We remark that $ L_{Q_3} \cap \mathcal{X}_2 = \{ A, B\}$, and $ L_{P_3} \cap \mathcal{Y}_2 = \{ C, D\}$. To finish, we have $ A \notin s_1 \cup s'_1 \cup s_2 \cup s'_2$, $ B \in s_1$, $ C \in r'_1 \cap r_2$, and $ D \in r_1$.

The unified critical locus $ \tilde{\mathcal{U}} \subseteq \mathbb{P}^3 \times \mathbb{P}^3 $ has dimension $ 1$ and degree $5$, and its defining ideal is computed according to Proposition \ref{prop-id-unif-crit-locus}.

Now, we consider the maps $ \tilde{\pi}_1: \tilde{\mathcal{U}} \to \mathcal{X}_3$ and $ \tilde{\pi}_2: \tilde{\mathcal{U}} \to \mathcal{Y}_3$. It is easy to check that
\begin{itemize}
\item[$(a)$] the points of $ L_{Q_3} $ different from $ A, B$ do not belong to $ \tilde{\pi}_1(\tilde{\mathcal{U}})$;
\item[$(b)$] the points of $ L_{P_3} $ different from $ C, D$ do not belong to $ \tilde{\pi}_2(\tilde{\mathcal{U}})$;
\item[$(c)$] $ \vert \tilde{\pi}_1^{-1}(Q) \vert = 1 $ for every $ Q \in \tilde{\mathcal{X}}$;
\item[$(d)$] $ \vert \tilde{\pi}_2^{-1}(P) \vert = 1 $ for every $ P \in \tilde{\mathcal{Y}}$;
\item[$(e)$] $ \tilde{\pi}_1^{-1}(A) = \{(A,(1:1:1:-2))\} = \tilde{\pi}_2^{-1}(1:1:1:-2)$;
\item[$(f)$] $ \tilde{\pi}_1^{-1}(B) = \{(B,C_{P_1})\} = \tilde{\pi}_2^{-1}(C_{P_1})$;
\item[$(g)$] $ \tilde{\pi}_1^{-1}(C_{Q_1}) = \{ (C_{Q_1}, D)\} = \tilde{\pi}_2^{-1}(D)$;
\item[$(h)$] $ \tilde{\pi}_1^{-1}(C_{Q_2}) = \{ (C_{Q_2}, C)\} = \tilde{\pi}_2^{-1}(C)$;
\item[$(i)$] $ \tilde{\pi}_1^{-1}(2:0:-1:-2) = \{((2:0:-1;-2),C_{P_2})\} = \tilde{\pi}_2^{-1}(C_{P_2})$.
\end{itemize}

Hence, $ \tilde{\pi}_1(\tilde{\mathcal{U}}) = \tilde{\mathcal{X}}$ and $ \tilde{\pi}_2(\tilde{\mathcal{U}}) = \tilde{\mathcal{Y}}$. Finally, it is possible to verify that $ \tilde{\pi}_2 \circ \tilde{\pi}_1^{-1}: \tilde{\mathcal{X}} \to \tilde{\mathcal{Y}} $ is an isomorphism, whose inverse is $ \tilde{\pi}_1 \circ \tilde{\pi}_2^{-1}$. In particular, $ \tilde{\pi}_i $ is not dominant in this case, for $i=1,2$.
\end{example}

In next example, we show that not every irreducible component of $ \mathcal{X} $ is birational to at least an irreducible component of $ \mathcal{Y}$.

\begin{example} \label{example-3-2} \rm
It is well known that every Bordiga surface is the blow--up of $ \mathbb{P}^2 $ at $ 10 $ general points. The complete linear system that embeds the blow--up in $ \mathbb{P}^4 $ is given by the plane quartics through the $10$ points. In \cite{Gimigliano}, the author proves that the Bordiga surface is singular when $ 4 $ points among the $10$ ones lie on a line. In \cite{bnt1}, the authors proved that every Bordiga surface is the critical locus of a couple of $3$ projections $ \mathbb{P}^4 \to \mathbb{P}^2$. They also provide an explicit way for getting the projection matrices from the syzygy matrix of the defining ideal of the $ 10 $ points in $ \mathbb{P}^2$. When one applies this construction to $10$ points, $4$ of which on a line, one gets a Bordiga surface $ \mathcal{X} $ with a singular point $P$ only. This point $P$ has the property that the rank of $ M_{\mathcal{X}}(P)$ is $6 = (k+n+1)-2$. The unified critical locus $ \mathcal{U} $ and the critical locus $ \mathcal{Y}$ can be easily constructed as explained in Propositions \ref{prop-id-unif-crit-locus} and \ref{prop-2-1}. It turns out that $ \mathcal{Y} $ is the union of a plane $ L $ and a degree $ 5 $ surface $ \mathcal{Y}_1$. The surface $ \mathcal{Y}_1 $ is a Castelnuovo surface, obtained as the blow--up of $ \mathbb{P}^2 $ at $ 7 $ simple and $ 1 $ double point. The fibre over the singular point $ P $ is $ \pi_1^{-1}(P) = \{P\} \times L$. Then, $ L $ is not birational to an irreducible component of $ \mathcal{X}$. Finally, $ \mathcal{X} $ and $ \mathcal{Y}_1 $ are birational. Indeed, from a side, they both are birational to $ \mathbb{P}^2$, from the other, $ \pi_2 \circ \pi_1^{-1} $ gives an explicit birational map.
\end{example}



\end{document}